\documentclass[11pt]{amsart}
\usepackage[utf8]{inputenc}
\usepackage[T1]{fontenc}
\usepackage[margin=3cm]{geometry}
\RequirePackage[colorlinks,citecolor=blue,urlcolor=blue]{hyperref}
\usepackage{comment,graphicx}

\newtheorem{thm}{Theorem}
\newtheorem{lemma}[thm]{Lemma}
\newtheorem{prop}[thm]{Proposition}

\newcommand{\p}{{\mathbb P}}
\newcommand{\e}{{\mathbb E}}
\newcommand{\D}{{\mathrm d}}
\newcommand{\bs}{\boldsymbol}
\newcommand{\ba}{{\bs \alpha}}

\newcommand{\R}{{\mathbb R}}

\newcommand{\1}[1]{\mbox{\rm\large  1}_{\{#1\}}}

\newcommand{\mat}[1]{\boldsymbol{\bs #1}}

\begin{document}
\title{On scale functions for L\'evy processes with negative phase-type jumps}
\author[J.\ Ivanovs]{Jevgenijs Ivanovs}
\begin{abstract}
We provide a novel expression of the scale function for a L\'evy processes with negative phase-type jumps.
It is in terms of a certain transition rate matrix which is explicit up to a single positive number.
A monotone iterative scheme for the calculation of the latter is presented and it is shown that the error decays exponentially fast.
Our numerical examples suggest that this algorithm allows to employ phase-type distributions with a hundred of phases, 
which is problematic when using the known formula for the scale function in terms of roots.  Extensions to other distributions, such as matrix-exponential and infinite-dimensional phase-type, can be anticipated.
\end{abstract}

\maketitle

\section{Introduction}

The theory of fluctuations of L\'evy processes with one-sided jumps is abundant in various identities and expressions, see the review papers~\cite{avram_review,scale_review}, the monographs~\cite{dkebicki2015queues,kyprianou} and references therein for a long list of formulas and applications.
These concern first passage times and overshoots, extremes, reflection and refraction, limiting distributions and distributions at exponential times, Poissonian observation, optimal control, and a great variety of other models and objectives.
Most of these expressions are in terms of a so-called scale function $W_q:\R_+\mapsto\R_+$ identified by its transform, where $q\geq 0$ is the killing rate of the underlying L\'evy process.
In queueing context, the distribution function of the workload at an independent exponential time (and at $\infty$) in a L\'evy-driven queue can be succinctly expressed in terms of~$W_q$, both in infinite and finite buffer cases.
We note that \textsc{Google Scholar} finds about  2070 articles containing `scale function' in the context of L\'evy processes as of today. 

Calculation of scale functions by inversion is a feasible but nontrivial task~\cite{scale_review}, and it may be prohibitive in a common scenario when a family of scale functions is needed for a large number of killing rates~$q$. 
Even though a generous number of (semi-) explicit examples of scale functions can be engineered~\cite{hubalek, scale_review}, arguably the most important explicit examples are given by processes with jumps of rational transform and so-called meromorphic processes~\cite{meromorphic}. Our focus is on the first class of processes.

We consider a spectrally-negative L\'evy process $(X_t)_{t\geq 0}$ with finite jump activity, which is often called perturbed Cram\'er-Lundberg risk process in actuarial science literature.
That is,
\begin{equation}\label{eq:X}X_t=dt+\sigma B_t-\sum_{i=1}^{N_t}C_i,\qquad t\geq 0,\end{equation}
where $B_t$ is a standard Brownian motion, $N_t$ is a Poisson process of rate $\lambda>0$, $C_i$ is a sequence of independent and identically distributed positive random variables, and all components are independent.
It is assumed that $\sigma\geq 0,d\in\R$ with $d>0$ when $\sigma=0$ to avoid monotone paths.
It is noted that certain series expansions of scale functions for such processes have been recently obtained in~\cite{lan_wil} (these are based on convolution powers of the jump distribution).
We assume, however, that $C_i$ has a phase-type (PH) distribution, that is, the distribution of the life-time of some transient continuous time Markov chain with finitely many states, often called phases.
These form a dense class of distributions on $\R_+$ (a proper subclass of distributions with rational transform) and lead to a profusion of tractable models in applied probability, see~\cite[Ch.\ 3]{APQ} or~\cite{ME}.

The scale function for a spectrally-negative L\'evy process with PH jumps can be expressed in terms of the zeros of a certain rational function, see~\cite[Prop.\ 2.1]{yamazaki} and \cite{scale_review}.
The number of these (possibly complex) roots is normally close to the number of phases.
Hence finding the roots may become problematic when the number of phases is substantial, which is a common scenario in applications.  
In this regard we note that~\cite{asmussen_laub} use 20 to 100 phases in their life insurance application.
For a survey on computational methods for two-sided L\'evy processes with PH jumps we refer to~\cite{asmussen_levy_martingales}.

An alternative approach is to use fluid embedding of PH jumps to arrive at a second order fluid flow model (Markov modulated Brownian motion) and to treat the given problem in that context, see~\cite{mordecki_PH,ph_asmussen_avram,asmussen_levy_martingales} and references in the latter.
There is a well-developed theory for such models~\cite{asmussen_fluid,iva_palm} and, in fact, a more general matrix-valued scale  function for an MMBM is given in~\cite[Ch.\ 7.7]{thesis} in terms of certain basic matrices.  
This, however, requires dealing with matrix calculus, partitioning of phases, and various further complications, not to mention a certain necessary experience. Moreover, such approach largely ignores the extensive literature on one-sided L\'evy processes and related models.

In this note we use some basic insights from the analysis of fluid flow models to establish an alternative expression of the scale function, see Theorem~\ref{thm:bm} and Theorem~\ref{thm:cpp} corresponding to $\sigma>0$ and $\sigma=0$.
These simple expressions are in terms of a certain transition rate matrix~$\mat G$, which is \emph{explicit up to a single positive number}.
In fact, it is fully explicit in the special case of $q=0$ and $\e X_1>0$.
Furthermore, we provide an iterative scheme yielding a monotone sequence of approximations of the unknown number, see Proposition~\ref{prop:bm} and Proposition~\ref{prop:cpp}, and demonstrate both theoretically and numerically its fast convergence.
This allows to employ PH distributions with hundreds of phases, while still having access to a plethora of results and expressions in terms of scale functions.
It is noted that the eigenvalues of the transition rate matrix are the above mentioned roots (not counting the non-negative one) and hence our iterative scheme can also be used to efficiently compute the roots.
Extensions to distributions with rational transforms and to (heavy-tailed) infinite-dimensional PH distributions~\cite{bladt_samorod} can be anticipated, but they require further investigation. 
Finally let us mention that in a closely related field of matrix analytic methods the use of (monotone) iterative schemes is standard and is normally preferred to the spectral method~\cite{ramaswami}.

\section{Preliminaries}
\subsection{Spectrally-negative L\'evy processes}
Let $(X_t)_{t\geq 0}$ be a general L\'evy process with no positive jumps, but not a process with a.s.\ decreasing paths.
The Laplace exponent of $X$ is denoted by $\psi(\theta)=\log \e e^{\theta X_1},\theta \geq 0$, and the first passage times (above $x$ and below $-x$) are given by
\[\tau_x^+=\inf\{t\geq 0:X_t>x\},\qquad \tau_{-x}^-=\inf\{t\geq 0:X_t<-x\},\qquad x\geq 0.\]
For $q\geq 0$ let $\Phi_q\geq 0$ be the right-most non-negative root of $\psi(\theta)=q$, which is known to satisfy the basic identity 
\[\e e^{-q\tau_x^+}=\p(\tau_x^+<e_q)=e^{-\Phi_q x},\qquad x\geq 0.\]
We write $e_q$ for an independent exponentially distributed random variable of rate $q\geq 0$ which is $\infty$ for $q=0$. 
This $e_q$ can be seen as the killing time of $X$, and thus $q\geq 0$ is just another parameter - the killing rate.

For every $q\geq 0$ there is a so-called scale function $W_q:[0,\infty)\mapsto [0,\infty)$, which is a continuous, non-decreasing function identified by its transform
\begin{equation}\label{eq:transform}\int_0^\infty e^{-\theta x}W_q(x)\D x=1/(\psi(\theta)-q),\qquad \theta>\Phi(q),\end{equation}
see~\cite[Thm.\ VII.8]{bertoin_book} or~\cite[Thm.\ 8.1]{kyprianou}.
The scale function is strictly positive for $x>0$ and it solves the basic two-sided exit problem:
\[\e(e^{-q\tau_x^+};\tau_x^+<\tau_{-y}^-)=W_q(y)/W_q(x+y),\qquad x,y\geq 0,x+y>0.\]
We refer to~\cite{avram_review,scale_review} for a long list of useful formulas based on~$W_q$.

Throughout the rest of this paper we assume that
\begin{equation}\label{eq:assumption}\tag{A1}
q>0\qquad\text{ or }\qquad \psi'(0)\neq 0,
\end{equation}
which merely excludes the case of a non-killed process $X$ with zero expectation. This case normally can be treated by taking the limit as $q\downarrow 0$.

\subsection{Application to queueing}
Here we briefly discuss L\'evy-driven queues and provide some basic formulas illustrating importance of $W_q$ in this domain. 
The material of this subsection is not used in the rest of the paper.
The workload process $(V_t)_{t\geq 0}$ of a queue driven by $-X_t$ (the classical case with positive jumps only) is defined by
\[V_t=v-X_t+\sup_{s\leq t}\{(v-X_s)^-\},\qquad v\geq 0,\]
where $v$ is the starting position and $x^-=-\min(x,0)$ is the negative part of~$x$, see~\cite{dkebicki2015queues}.
Under the stability condition $\mu:=\e X_1>0$ the workload process $V_t$ admits a weak limit $V_\infty$ which, according to the classical duality relation,  is given by
$\sup\{-X_t:t\geq 0\}$. Now
\[\p(V_\infty\leq x)=\mu W_0(x),\qquad x\geq 0,\]
see~\cite[(8.15)]{kyprianou}.
Hence $W_0$  is proportional to the distribution function of the stationary workload in the case when~$\e X_1>0$.
This can also be readily verified by taking the Laplace transform and using the generalized Pollaczek-Khinchine formula together with~\eqref{eq:transform}.

Distribution function of $V$ at an independent exponential time~$e_q$ can be derived from~\cite[Thm.\ 8.11]{kyprianou}:
\[\p(V_{e_q}\in \D x)=q\Big(\frac{1}{\Phi_q}e^{-\Phi_q v}W'_q(x)-W_q(x-v)\Big)\D x,\qquad x>0\]
together with a point mass $\p(V_{e_q}=0)=\frac{q}{\Phi_q}e^{-\Phi_q v}W_q(0)$, where $W_q$ is 0 for negative arguments.
There are many more useful formulas concerning L\'evy-driven queues and the ruin theory analogues scattered in the literature.

\subsection{A convenient representation of the scale function}
Define the first hitting time of a level $x\in \R$ by
\[\tau_{\{x\}}=\inf\{t>0:X_t=x\},\]
which a.s.\ coincides with $\tau_x^+$ for $x>0$.
The following representation of the scale function is useful in various contexts, but is not widely known:
\begin{equation}\label{eq:scale_alt}W_q(x)=\frac{1}{\psi'(\Phi_q)}\Big(e^{\Phi_qx}-\p(\tau_{\{-x\}}<e_q)\Big),\qquad x\geq 0,\end{equation}
see~\cite[Eq.\ (12)]{iva_palm} and also~\cite[Thm.\ 1]{pistorius}, \cite[Eq.\ (95)]{scale_review} for some related formulas.
Due to assumption~\eqref{eq:assumption} the denominator $\psi'(\Phi_q)\neq 0$.
This formula will serve as a basis for deriving our result.

For completeness let us mention that~\eqref{eq:scale_alt}  can be phrased in terms of the local time $L_t$ at zero, see~\cite[Ch.\ V]{bertoin_book} for the definition. 
In the case~\eqref{eq:X} with $\sigma=0$ the scaled local time $d L_t$ simply counts the epochs when the level~0 is hit.
Now $\e L_{e_q}=1/\psi'(\Phi_q)$ is the expected local time of the killed process and,  by the additive property of $L_t$ we have
\[e^{-\Phi_qx}W_q(x)=\e L_{e_q}-\p(\tau_x^+<e_q)\p(\tau_{\{-x\}}<e_q)\e L_{e_q}=\e L_{\tau_x^+\wedge e_q}=\e L_{\tau_{-x}^-\wedge e_q},\qquad x\geq 0.\]
In words, this is the expected local time at 0 of the killed process collected up to either $\tau_x^+$ or~$\tau_{-x}^-$. 

\subsection{Phase-type jumps}
Consider a continuous time Markov chain on a set of $n$ transient states with an initial distribution $\ba$ and $n\times n$ transition rate matrix $\mat{T}$.
Here $\ba$ is a row vector with $n$ non-negative elements summing up to~1, and $\mat{t}=-\mat{T1}\geq \mat 0$ is a column vector with $n$ elements giving the killing rates; we write $\mat{1}$ and $\mat{0}$ for the column vectors of ones and zeros, respectively.
The distribution of the life-time of this Markov chain is denoted by ${\rm PH}(\ba,\mat{T})$. 
Without loss of generality we assume that $\ba,\mat{T}$ are such that the Markov chain has a positive probability to visit any state.

The density of ${\rm PH}(\ba,\mat{T})$ has a matrix exponential form $f(x)=\ba e^{\mat{T} x}\mat{t},x>0$ and its transform is a rational function
\[\int_0^\infty e^{-\theta x}f(x)\D x=\ba (\theta\mat{I}-\mat{T})^{-1}\mat{t}=\frac{P(\theta)}{Q(\theta)},\qquad \theta\geq 0.\]
Here $P$ and $Q$ are polynomials of degree $p-1$ and $p$, respectively, with no common zeros in~$\mathbb C$. It must be that $1\leq p\leq n$, and in the case $p=n$ we say that the representation $(\ba,\mat{T})$ is minimal.
It is noted that one can always provide a minimal matrix exponential representation, whereas a minimal PH representation need not exist~\cite[Sec.\ 4.2]{ME}.
Let us also point out that the zeros of $Q$ are the eigenvalues of $\mat{T}$ and the latter have negative real parts.

In the following we assume that our L\'evy process $X$ has the form given in~\eqref{eq:X} with $C_i\sim{\rm PH}(\ba,\mat{T})$.
Equivalently,
\begin{equation}\label{eq:psi}\tag{A2}\psi_q(\theta):=\psi(\theta)-q=\frac{1}{2}\sigma^2\theta^2+d\theta+\lambda(\ba (\theta\mat{I}-\mat{T})^{-1}\mat{t}-1)-q=\frac{\tilde P(\theta)}{Q(\theta)}, \qquad\theta\geq 0.\end{equation}
Note that $\tilde P$ and $Q$ have no common zeros and the degree of $\tilde P$ is $p+2-\1{\sigma=0}$.
Finally, $\e C_i=\ba(-\mat{T})^{-1}\mat{1}$ and so 
$\psi'(0)=\e X_1=d+\lambda\ba\mat{T}^{-1}\mat{1}$.

\subsection{The scale function in terms of roots}
Recall that $\psi_q(\theta)$ has $p+2-\1{\sigma=0}$ zeros in~$\mathbb C$ counting multiplicities for any $q\geq 0$.
Let $\mathcal Z_q$ be the set of these zeros excluding the single zero at~$\Phi_q$:
\begin{equation}\label{eq:Zq}\mathcal Z_q=\{z\in\mathbb C: \psi_q(z)=0,z\neq \Phi_q\}.\end{equation}
For any $z\in \mathcal Z_q$ it must be that $\Re(z)\leq 0$, which follows from the standard properties of~$\psi$.
Assuming that the zeros in $\mathcal Z_q$ are simple, there is the identity
\begin{equation}\label{eq:Wroots}
W_q(x)=\frac{e^{\Phi_qx}}{\psi'(\Phi_q)}+\sum_{z\in \mathcal Z_q}\frac{e^{z x}}{\psi'(z)},\qquad x\geq 0,
\end{equation}
and the number of elements of $\mathcal Z_q$ is $p+\1{\sigma>0}$.
Indeed, by taking transform we obtain a partial fraction decomposition of $1/\psi_q(\theta)$.
This result and also its less neat version in the case of multiple zeros can be found in~\cite[Prop.\ 2.1]{yamazaki} and in~\cite[Sec.\ 5.4]{scale_review}.
In the case $q=0,\psi'(0)<0$ the set $\mathcal Z_q$ contains $0$, which is somewhat unclear in the above cited works.

Recall that for a minimal PH representation $(\ba,\mat{T})$ we have $p=n$, and so we need to find $n+\1{\sigma>0}$ zeros of the rational function in~\eqref{eq:psi} in the left half of the complex plane (computing $\Phi_q$ is trivial thanks to convexity of~$\psi_q$).
As discussed before, in some applications this number may exceed 100, and thus general root finding methods may fail to find all the zeros or may require certain adaptations. 
One may also use efficient numerical procedures~\cite{bini} for locating zeros of polynomials, but that requires to pick out the numerator or to multiply by $\det(\theta\mat I-\mat T)$.
Importantly, such methods require careful use of multiprecision due to numerical instability.

\section{The scale function in terms of a transition rate matrix}
In view of~\eqref{eq:scale_alt} we aim to find a simple expression of the hitting probability $\p(\tau_{\{-x\}}<e_q)$ for $x\geq 0$.
This task can be achieved by fluid embedding: Consider a continuous time Markov chain $J_t$ on a set of $n+1$ states (phases) with a transition rate matrix 
\[\begin{pmatrix}
-\lambda-q &\lambda\ba\\
\mat{t}& \mat{T}
\end{pmatrix},\]
and assume that the level process $Y_t$ starts at 0 and evolves as an independent linear Brownian motion $\sigma B_t-d t$ when $J_t=1$ and as a linear drift $t$ with unit slope when $J_t\neq 1$.
This can be compactly stated as 
\[Y_0=0,\qquad \D Y_t=\1{J_t=1}(\sigma\D B_t-d\D t)+\1{J_t\neq 1}\D t,\]
where the process $Y$ is killed upon termination of $J$. The bivariate process $(Y_t,J_t)$ is a special case of a well-studied Markov modulated Brownian motion (MMBM), see~\cite{asmussen_fluid,ivanovs_MMBM}.
Note that the process $Y$ with time intervals $J_t\neq 1$ deleted has the law of $-X$ process killed at~$e_q$.
Thus, for $x\geq 0$ the probability $\p(\tau_{\{-x\}}<e_q)$ coincides with the probability of $(Y,J)$ hitting $(x,1)$ given $J_0=1$. 

Letting $\varsigma_x=\inf\{t\geq 0:Y_t>x\}$ we see that $(J_{\varsigma_x})_{x\geq 0}$ is a Markov chain, which may or may not enter state~1 according to $\sigma>0$ and $\sigma=0$.
We treat these two cases separately.

\subsection{Brownian component is present}
Assume that $\sigma>0$ and let $\mat{G}$ be $(n+1)\times(n+1)$ transition rate matrix of the first passage Markov chain~$(J_{\varsigma_x})_{x\geq 0}$.
Now \[\p(\tau_{\{-x\}}<e_q)=\sum_j\p(J_{\varsigma_x}=j|J_{\varsigma_0}=1)\nu_j=\mat{e}_1e^{\mat{G}x}\mat{\nu},\]
where $\mat{e}_1=(1,0,\ldots,0)$ and $\mat{\nu}$ is an $(n+1)$ column vector with $\nu_j$ being the probability that given $J_0=j$ the process $Y$ ever hits level 0, which must be in phase~$1$.
Clearly, $\nu_1=1$ and the other elements are given by
\[\int_0^\infty e^{\mat{T}x}\mat{t}\p(\tau_x^+<e_q)\D x=\int_0^\infty e^{\mat{T}x}e^{-\Phi_q x}\D x\mat{t}=(\Phi_q\mat{I}-\mat{T})^{-1}\mat{t}.\]
Furthermore, a linear Brownian motion hits $(0,\infty)$ immediately a.s.\ and we see from the definition of the process~$(Y,J)$ that
\begin{equation}\label{eq:Gnu}\mat{G}=\begin{pmatrix}
-a &\mat{b}\\
\mat{t}& \mat{T}
\end{pmatrix}, \qquad \mat{\nu}=\begin{pmatrix}
1\\
(\Phi_q\mat{I}-\mat{T})^{-1}\mat{t}
\end{pmatrix}\end{equation}
for some number $a>0$ and $n$-vector $\mat{b}>\mat{0}$ such that $\mat{b1}\leq a$.
It is thus left to characterize $a$ and~$\mat{b}$.
Before doing so we note that $a=\mat{b1}$ when $q=0,\e X_1\leq 0$ and $a>\mat{b1}$ otherwise, which corresponds to the first passage chain $J_{\varsigma_x}$ being recurrent and transient, respectively.
Note that $\mat \nu,a,\mat b$ depend on the killing rate $q\geq 0$.

The matrix~$\mat G$ is a fundamental object in the analysis of an MMBM, and it is normally computed using an iterative scheme based on the characterizing equation~\eqref{eq:G} below. 
In this setting we have an almost explicit formula which, as we next show, depends on a single unknown $a>0$ being a certain fixed point.

\begin{thm}\label{thm:bm}
Assuming~\eqref{eq:assumption} and \eqref{eq:psi} with $\sigma,\lambda>0,d\in \R,q\geq 0$ we have
\[W_q(x)=\frac{1}{\psi'(\Phi_q)}\Big(e^{\Phi_qx}-\mat{e}_1e^{\mat{G}x}\mat{\nu}\Big),\qquad x\geq 0,\]
where the transition rate matrix $\mat G$ and a non-negative vector $\mat{\nu}$ are given in~\eqref{eq:Gnu}.
For $q=0,\e X_1>0$ we have $a=\frac{2d}{\sigma^2}>0$ and $\mat{b}=\frac{2\lambda}{\sigma^2}\ba(-\mat{T})^{-1}>\mat 0$,
and otherwise $\mat{b}\geq \mat 0$ and $a\geq\mat{b1}$ are uniquely characterized by
\begin{equation}\label{eq:ab}\frac{1}{2}\sigma^2(a^2+\mat{b}\mat{t})-da=\lambda+q,\qquad \frac{1}{2}\sigma^2(\mat{b}\mat{T}-a\mat{b})+d\mat{b}=-\lambda\ba\end{equation}
which implies 
\begin{equation}\label{eq:interval}2d^+/\sigma^2<a<\big(d+\sqrt{d^2+2\sigma^2(\lambda+q)}\big)/\sigma^2.\end{equation} 
\end{thm}
\begin{proof}
In view of~\eqref{eq:scale_alt} and the above considerations, it is only required to characterize $a>0$ and $\mat b$. 
Note that the transition rate matrix of $J$ is irreducible, since we have excluded certain degenerate PH representations.
From the basic MMBM theory, see e.g.~\cite{asmussen_fluid}, \cite[Thm.\ 2]{jordan}, \cite[Cor.\ 4.15]{thesis} (the latter allows for phase-dependent killing), we have
the matrix equation
\begin{equation}\label{eq:G}\frac{1}{2}\sigma^2\mat\Delta_{(1,0,\ldots,0)} \mat{G}^2-\mat\Delta_{(-d,1,\ldots,1)}\mat{G}+\begin{pmatrix}
-\lambda-q &\lambda\ba\\
\mat{t}& \mat{T}
\end{pmatrix}=\mat{O},\end{equation}
where $\mat \Delta_{\mat d}$ stands for a diagonal matrix with $\mat{d}$ on the diagonal. 
Observe that the first row readily provides the stated identities in~\eqref{eq:ab}, whereas the other rows yield the above observed form of~$\mat G$.
Furthermore, the above equation uniquely characterizes $\mat G$ given that $\mat G$ is a recurrent transition rate matrix $(a=\mat{b1})$ when $q=0,\e X_1\leq 0$ and transient $(a>\mat{b1})$ otherwise.

The first equation in~\eqref{eq:ab} can be replaced by
\begin{equation}\label{eq:ab2}(\frac{1}{2}\sigma^2 a-d)(a-\mat{b1})=q,\end{equation}
which follows from $\frac{1}{2}\sigma^2(\mat{b}\mat{t}+a\mat{b1})-d\mat{b1}=\lambda$ obtained by right-multiplying the second equation by $\mat{1}$.
For $q=0,\e X_1>0$ we must have $a>\mat{b1}$  and so $a=2d/\sigma^2$ which readily yields the above stated $\mat b$.
For $q>0$ we see from~\eqref{eq:ab2} that $a>\mat{b1}$, as required. 
Finally, assume $q=0,\e X_1=d-\lambda\ba(-\mat T)^{-1}\mat 1<0$. If $a=2d/\sigma^2\geq 0$ then from the second equation in~\eqref{eq:ab} we find $\mat b\mat 1>2d/\sigma^2=a$. This contradiction means that $a=\mat{b1}$ as required, and the characterization is now complete.
In the latter case $a<2d/\sigma^2$ is impossible which can be shown directly with some effort, or by appealing to continuity of $a$ in $q\downarrow 0$. 
Now the bounds in~\eqref{eq:interval} easily follow, 
 and the proof is complete.
 \end{proof}

Consider the explicit case $q=0,\e X_1>0$ and note that $\Phi_q=0$ and $\mat \nu=\mat 1$. Thus we have
\[1-\mu W_0(x)=\mat e_1 e^{\mat G x}\mat 1,\qquad x\geq 0\]
which can be identified with the ruin probability $\p(\inf_{t\geq 0}\{x+X_t\}<0)$. Note also that the rate $a=2d/\sigma^2$ does not depend on $\lambda$ and the PH distribution of jumps. 
An analogous formula in the case of no Brownian component is well known, see~\cite[Cor.\ IX.3.1]{ruin} and the following subsection.

In the above explicit case $(a,\mat b)$ solve~\eqref{eq:ab}, but we do not claim that they are thus characterized without additionally assuming that $a>\mat{b 1}$.
In fact, one can provide simple examples with an additional solution satisfying $a=\mat{b1}>0$.
In the other cases 
\[\mat b=\frac{2\lambda}{\sigma^2} \ba((a-2d/\sigma^2)\mat{I}-\mat{T})^{-1}\] can be easily computed if the value of $a$ is given, and so $a$ is a single unknown. We provide a monotone iterative scheme to compute this unknown in Section~\ref{sec:iterative}.

\subsection{No Brownian component}
Here we assume that $\sigma=0$ and $d>0$.
Now the first passage chain $(J_{\varsigma_x})_{x\geq 0}$ lives on the states $2,\ldots,n+1$, and we let $\mat G$ be its $n\times n$ transition rate matrix.
Moreover, let $\mat \pi$ be a row vector with $n$ elements denoting the distribution of $J_{\varsigma_0}$ given $J_0=1$.

\begin{thm}\label{thm:cpp}
Assuming~\eqref{eq:assumption} and \eqref{eq:psi} with $\sigma=0,d,\lambda>0,q\geq 0$ we have
\[W_q(x)=\frac{1}{\psi'(\Phi_q)}\Big(e^{\Phi_qx}-\mat\pi e^{\mat{G}x}\mat{\nu}\Big),\quad x\geq 0,\qquad\mat G=\mat T+\mat{t\pi},\quad \mat \nu=(\Phi_q\mat{I}-\mat{T})^{-1}\mat{t}.\]
For $q=0,\e X_1>0$ we have
$\mat\pi=\frac{\lambda}{d}\ba(-\mat T)^{-1}>\mat 0$, and otherwise $\mat \pi\geq 0$ with $\mat{\pi 1}\leq 1$ is uniquely characterized by
\begin{equation}\label{eq:pi}\mat\pi\Big((\lambda+q-d\mat{\pi t})\mat I-d\mat T\Big)=\lambda\ba\end{equation}
which implies $\mat{\pi t}\in(0,\lambda/d)$. 
\end{thm}
\begin{proof}
Note that $\p(\tau_{\{-x\}}<e_q)=\mat\pi e^{\mat{G}x}\mat{\nu}$ with $\mat \nu$ of the stated form. Furthermore, the form of $\mat G$ follows from the very simple structure of the process $(Y,J)$.
More precisely, the rates are given by $\mat T$ plus the rates of transition via the state~1. In view of~\eqref{eq:scale_alt} it is left to identify~$\mat\pi$.
From the MMBM theory~\cite[Cor.\ 4.15]{thesis} we know that
\begin{equation*}-\mat\Delta_{(-d,1,\ldots,1)}\begin{pmatrix}\mat \pi\\ \mat I\end{pmatrix}\mat{G}+\begin{pmatrix}
-\lambda-q &\lambda\ba\\
\mat{t}& \mat{T}
\end{pmatrix}\begin{pmatrix}\mat \pi\\ \mat I\end{pmatrix}=\mat{O}
\end{equation*}
uniquely characterizes the transition rate matrix~$\mat G$ and vector $\mat\pi$ under the transience/recurrence requirement, which is $\mat{\pi 1}<1$ and $\mat{\pi 1}=1$, respectively.
The first row gives
\[d\mat\pi\mat G-(\lambda+q)\mat\pi+\lambda\ba=\mat{0},\]
and the second gives the above observed form of~$\mat G$, and thus~\eqref{eq:pi} follows.

Again we may right-multiply~\eqref{eq:pi} by $\mat 1$ to derive a further useful equation:
\begin{equation}\label{eq:pi2}(\lambda-d\mat{\pi t})(1-\mat{\pi 1})=\mat{\pi 1}q.\end{equation}
For $q=0,\e X_1>0$ we must have $\mat{\pi 1}<1$ and so $\mat{\pi t}=\lambda/d$ and the expression for $\mat \pi$ follows.
For $q>0$ we have  $\mat{\pi 1}<1$ as required. If $q=0,\mat{\pi t}=\lambda/d$ then $\mat\pi=\frac{\lambda}{d}\ba(-\mat T)^{-1}$ showing that $d(1-\mat{\pi 1})=\e X_1\geq 0$.
Thus for $\e X_1<0$ we must have $\mat{\pi 1}=1$, as required. The characterization result is now complete, and the bounds on $\mat\pi\mat t$ are obvious apart from the case $q=0,\e X_1<0$.
The latter case can be treated by taking $q\downarrow 0$. 
\end{proof}

The explicit case $q=0,\e X_1>0$ results in a well-known identity~\cite[Cor.\ IX.3.1]{ruin} for the ruin probability:
\[\p(\inf_{t\geq 0}\{x+X_t\}<0)=1-\mu W_0(x)=\mat \pi e^{\mat G x}\mat 1,\qquad x\geq 0,\]
where $\mat\pi=\frac{\lambda}{d}\ba(-\mat T)^{-1}$ and $\mat G=\mat T+\mat{t\pi}$.
Similarly to the Brownian case, there may exist another solution $\mat\pi$ of~\eqref{eq:pi} with $\mat{\pi 1}=1$.
In the other cases the vector $\mat \pi$ can be easily computed if the value of the number $\mat \pi\mat t$ is provided.

\subsection{Relation to the formula in terms of roots}
Consider the expressions of $W_q$ in Theorem~\ref{thm:bm} and in Theorem~\ref{thm:cpp}.
Given that the corresponding transition rate matrix $\mat G$ is diagonalizable these expressions can be rewritten as $\sum_i c_i e^{\zeta_i x}$, where $\zeta_i$ runs through $\Phi_q$ and the eigenvalues of $\mat G$.
This gives an expression analogous to~\eqref{eq:Wroots}, but even more can be said.

\begin{lemma}\label{lem:zeros}
Assume~\eqref{eq:assumption} and minimality of PH representation $(\ba,\mat T)$. Then the roots in~\eqref{eq:Zq} are simple iff the $n+\1{\sigma>0}$ eigenvalues of $\mat G$ are distinct, in which case the two sets coincide.
\end{lemma}
\begin{proof}
Take the transform of the expression in Theorem~\ref{thm:bm} and in Theorem~\ref{thm:cpp} and recall that it must coincide with~\eqref{eq:transform}.
Apply analytic continuation and use the fact that $\psi(\theta)-q$ has $n+1+\1{\sigma>0}$ zeros.
\end{proof}


\section{Monotone iterative schemes}\label{sec:iterative}
In this section we propose an efficient algorithm to compute the unknowns $a,\mat b$ characterized by~\eqref{eq:ab}, as well as an algorithm for $\mat\pi$ characterized by~\eqref{eq:pi}.
As mentioned above, a known $a$  yields $\mat b$ and a known $\mat{\pi t}$ yields $\mat \pi$, and so there is a single unknown number underlying the new representations of the scale function.
In the MMBM theory there are various iterative schemes for calculation of $\mat G$ and the associated initial distributions, see~\cite{asmussen_fluid}.
Our present problem, however, has a very simple structure suggesting some particular schemes, which we discuss and analyze below.

Let us first introduce some terminology used in numerical analysis.
Consider an iterative scheme $x_{n+1}=f(x_n)$ for a differentiable function $f:\R\mapsto\R$ and some starting~$x_0$.
Assume that $x_n\to x\in\R$ where $x=f(x)$ is necessarily a fixed point, and observe that
\begin{equation}\label{eq:rate}\frac{x_{n+1}-x}{x_n-x}=\frac{f(x_n)-f(x)}{x_n-x}\to f'(x)\in[-1,1].\end{equation}
If $\mu:=|f'(x)|\in (0,1)$ then it is common to say that $x_n$ converges Q-linearly with rate $\mu$, see~\cite{numerics}.
One can easily check that in this case $\log |x_n-x|/n\to \log \mu$, 
that is, the error $|x_n-x|$ decays as $\mu^n$ in the logarithmic sense. 

\subsection{Brownian component is present}
Reconsider Theorem~\ref{thm:bm} in the non-explicit case and recall that $a>2d/\sigma^2$. 
For any $a_0> 2d/\sigma^2$ we define $(\mat b_n,a_n)_{n\geq 1}$ recursively from~\eqref{eq:ab} as follows:
\[ \mat{b}_n=\frac{2\lambda}{\sigma^2} \ba((a_{n-1}-2d/\sigma^2)\mat{I}-\mat{T})^{-1},\qquad a_n=\big(d+\sqrt{d^2+\sigma^2(2\lambda+2q-\sigma^2\mat{b}_{n}\mat{t})}\big)/\sigma^2.\]
We now show that it is well-defined and converges monotonically to $(a,\mat{b})$. 
\begin{prop}\label{prop:bm}
Assume the conditions of Theorem~\ref{thm:bm} and that $q>0$ or $\e X_1<0$. 
For any $a_0>2d/\sigma^2$ the sequence $(a_n,\mat{b}_n)_{n\geq 1}$ is well-defined, satisfies the bounds in~\eqref{eq:interval}, and exhibits monotone convergence to $(a,\mat b)$:
\[a_n\uparrow a, \mat{b}_n\downarrow \mat{b}\quad\text{ for  }a_0\leq a, \qquad a_n\downarrow a, \mat{b}_n\uparrow \mat{b}\quad\text{ for  }a_0\geq a.\]

Moreover, the sequence $a_n$ converges $Q$-linearly with the rate
\[\frac{\mat b((a-2d/\sigma^2)\mat{I}-\mat{T})^{-1}\mat t}{2a-2d/\sigma^2}\in(0,1)\]
which is decreasing in $q\geq 0$.
\end{prop}
\begin{proof}
Assume that $a_{n-1}> 2d/\sigma^2$ which is true for $n=1$.
Then $\mat{T}-(a_{n-1}-2d/\sigma^2)\mat{I}$ is a transient transition rate matrix implying that 
$\mat{b}_n$ is well defined and has strictly positive elements.
Moreover, we observe that $\sigma^2\mat b_n\mat t<2\lambda\ba(-\mat T)^{-1}\mat t=2\lambda$ and hence $a_n> 2d/\sigma^2\vee 0$.
Interpreting $a_{n-1}-2d/\sigma^2$ as a killing rate we see that $\mat b_m< \mat b_n$ iff $a_{m-1}> a_{n-1}$, in which case $a_m>a_n$.
Thus $a_n>0$ is monotone and so is $\mat b_n>\mat 0$ but in the opposite way.
Hence they must have a finite limit $a^*\geq 0,\mat b^*\geq \mat 0$ which solves~\eqref{eq:ab}.

For $q>0$ we also find from~\eqref{eq:ab2} that $a^*\geq \mat b^*\mat 1$ and hence characterization result in Theorem~\ref{thm:bm} can be applied. 
In the case $q=0,\e X_1<0$ we only need to show that $a^*\neq 2d/\sigma^2$, and so we may assume that $d\geq 0$. Recall that $a_{n-1}>2d/\sigma^2$ and hence it is sufficient to show that $a_{n-1}=2d/\sigma^2+\epsilon$ implies $a_{n}>a_{n-1}$ for all small enough $\epsilon>0$.
That is, we need to show that 
\[\sqrt{d^2+\sigma^2(2\lambda-2\lambda\ba(\epsilon \mat I-\mat T)^{-1}\mat t)}> \epsilon\sigma^2+d,\]
which is equivalent to
\[2\lambda(1-\ba(\epsilon \mat I-\mat T)^{-1}\mat t)/\epsilon >\epsilon\sigma^2+2d.\]
The left hand side converges to the derivative $2\lambda\ba(-T)^{-1}\mat 1$ as $\epsilon\downarrow 0$ which exceeds $2 d$ by assumption of the negative drift, completing the proof of the first part.

It is left to analyze the rate of convergence of~$a_n$, and according to~\eqref{eq:rate} we consider
\[f(a)=\big(d+\sqrt{d^2+\sigma^2(2\lambda+2q-2\lambda\ba((a-2d/\sigma^2)\mat{I}-\mat{T})^{-1}\mat{t})}\big)/\sigma^2.\] Differentiating while using $a=f(a)$ we obtain
\[f'(a)=\frac{\lambda\ba((a-2d/\sigma^2)\mat{I}-\mat{T})^{-2}\mat{t}}{\sigma^2 a-d}=\frac{\mat b((a-2d/\sigma^2)\mat{I}-\mat{T})^{-1}\mat{t}}{2a-2d/\sigma^2}.\]
The numerator is upper bounded by $\mat b(-\mat T)^{-1}\mat t=\mat b\mat 1\leq a$, whereas the denominator exceeds~$a$.
Hence $f'(a)<1$ and its strict positivity follows from $\mat b\mat 1>0$.  Finally, $a$ is an increasing function of $q$ and thus the convergence rate is a decreasing function. 
 \end{proof}

It is important to note that the starting value $a_0=2d/\sigma^2$ can be used for $q>0$, whereas for $q=0$ it will result in a constant sequence $a_n=2d/\sigma^2$ irrespective of the drift and the true solution.
Observe that an alternative iterative scheme can be obtained by expressing $a$ from~\eqref{eq:ab2}.
This, however, does not yield a monotone sequence and it also exhibits a much slower convergence in our numerical examples below.

The number of iterations depends on the starting position $a_0$. For large $q>0$ we expect $\mat b$ to be close to $\mat 0$ and so a good starting position 
is given by the upper bound in~\eqref{eq:interval} corresponding to $\mat b=\mat 0$.
If the scale matrix is computed for a number of different~$q$ then we may start from the largest value and use the last solution as initial $a_0$ for the following~$q$. 
In general, as a rule of thumb one may use the midpoint of the interval in~\eqref{eq:interval}.

\subsection{No Brownian component}
Reconsider Theorem~\ref{thm:cpp}. In this case an obvious recursion is
\[\mat\pi_n=\lambda\ba\Big((\lambda+q-d\mat{\pi}_{n-1}\mat{t})\mat I-d\mat T\Big)^{-1}.\]
Here we only need to specify $\mat\pi_0$ up to the value of $\mat\pi_0\mat t$.
\begin{prop}\label{prop:cpp}
Under the assumptions of Theorem~\ref{thm:cpp} for any $\mat\pi_0\mat t< \lambda/d$ the sequence $(\mat{\pi}_n)_{n\geq 1}$ is well-defined, satisfies $\mat 0<\mat\pi_n\mat t<\lambda/d$, and 
converges monotonically to $\mat\pi$.

Moreover, the sequence $\mat \pi_n\mat t$ converges Q-linearly with the rate
\[\mat \pi\Big(((\lambda+q)/d-\mat \pi\mat t)\mat I-\mat T\Big)^{-1}\mat t\in (0,1)\]
which is decreasing in $q\geq 0$.
\end{prop}
\begin{proof}
The matrix under inverse is a transition rate matrix when $\mat\pi_{n-1}\mat t< \lambda/d$, and then it is easy to see that this condition is preserved in the sequence.
Moreover, $\mat\pi_{n+1}>\mat \pi_{n}$ iff $\mat\pi_{n}\mat t>\mat \pi_{n-1}\mat t$.
Hence the sequence $\mat\pi_n$ is monotone and thus has a limit $\mat \pi^*$ solving~\eqref{eq:pi}.

For $q>0$ from~\eqref{eq:pi2} we find that $\mat{\pi}^*\mat{1}<1$ and the characterization result completes the proof. For $q=0,\e X_1<0$ we only need to show that $\mat\pi^*\mat t\neq \lambda/d$.
Similarly to the Brownian case it is sufficient to assume that $\mat\pi_{n-1}\mat t=\lambda/d-\epsilon$ and to show that $\mat \pi_n\mat t<\mat \pi_{n-1}\mat t$ for all small enough $\epsilon>0$.
That is we need to show
\[\lambda\ba(\epsilon\mat I-\mat T)^{-1}\mat t<\lambda-\epsilon d,\]
which readily follows from $\ba(-\mat T)^{-1}\mat 1>d$.

It is left to identify the rate and according to~\eqref{eq:rate} we need to consider the derivative of
\[f(x)=\frac{\lambda}{d}\ba\Big(((\lambda+q)/d-x)\mat I-\mat T\Big)^{-1}\mat t\]
at $x=\mat \pi\mat t$, which using~\eqref{eq:pi} evaluates to the stated expression. Finally, $\mat\pi$ is decreasing in $q$ and thus the convergence rate has the same property.
\end{proof}

As a rule of thumb one may use the starting value $\mat\pi_0\mat t=\lambda/(2d)$.

\section{Numerical illustrations}
Numerical experiments in this section are performed using \textsc{Wolfram Mathematica 11} and the most straightforward implementation of required procedures. 
We choose $n=50$ and a PH distribution of Coxian type which is popular in applications, see~\cite{asmussen_laub}.
That is, the respective PH chain starts in phase 1 and it may only jump to the following phase or to terminate.
The parameters are sampled randomly: inverse standard uniforms for the rates out of each phase and $(0,0.9)$ uniforms for the killing probabilities. On average this construction results in expected value of about~$1$.
Our particular sample has expectation $1.77$ and its density is depicted in Figure~\ref{fig:zeros}.
Furthermore, we choose $\sigma=d=\lambda=1$  resulting in $\e X_1=-0.77$ and consider $q\in\{0,0.1,1\}$.

Let us recall the two methods of computing the scale function $W_q(x)$:
\begin{itemize}
\item[(Old):] Find the zeros of $\theta\mapsto\psi(\theta)-q$ in $\mathbb C$ and use formula~\eqref{eq:Wroots}.
\item[(New):] Form the transition rate matrix $\mat G$ and the vector $\mat \nu$ and combine them into $W_q(x)$ as specified in Theorem~\ref{thm:bm}.
This requires computing $(a,\mat b)$ which is done via the monotone iterative scheme in Proposition~\ref{prop:bm}. 
We stop when $|a_n-a_{n-1}|<10^{-5}$ and return $(a_n,\mat b_n)$ as an approximation of $(a,\mat b)$. Convergence of this scheme is investigated below, and by default we start at the midpoint of the interval in~\eqref{eq:interval}.
\end{itemize}
Note that the second method requires the non-negative root $\Phi_q$, which is trivial to compute thanks to convexity of $\psi(\theta),\theta\geq 0$. 
Furthermore, this root is needed by various fluctuation identities anyway.

\subsection{Comparison of the two methods}
\begin{figure}[h!]
\centering
\includegraphics[width=0.45\textwidth]{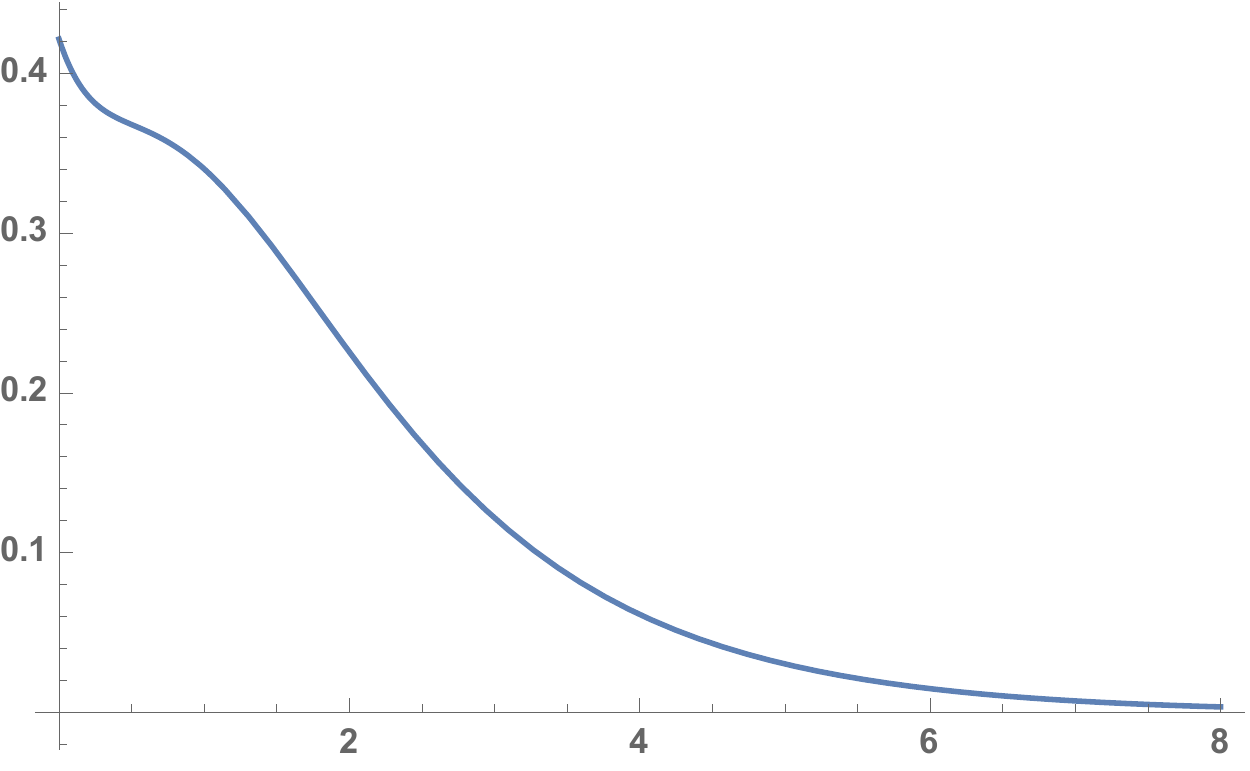}
\includegraphics[width=0.45\textwidth]{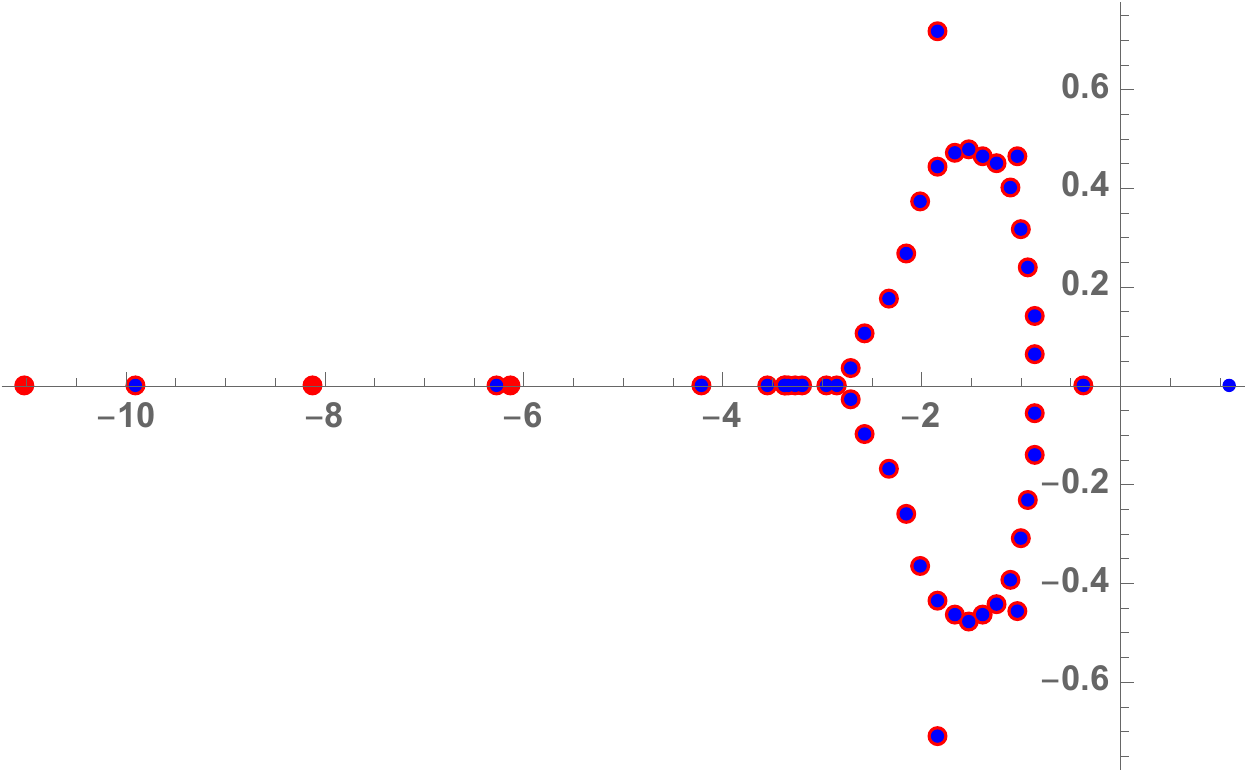}
\caption{Left: PH density. Right: The complex plane with the detected zeros of $\psi(\theta)-1$ in blue and the eigenvalues of $\mat G$ in red.}
\label{fig:zeros}
\end{figure}
Our first illustration concerns the correspondence between the zeros of $\psi(\theta)-q$ and the eigenvalues of $\mat G$ in Theorem~\ref{thm:bm} for $q=1$, see also Lemma~\ref{lem:zeros}. 
We call a generic \textsc{Solve} function and find 47 zeros, whereas the maximal number is 52. 
Additionally, we have tried various other \textsc{Mathematica} procedures including symbolic manipulation routines and \textsc{NRoots}, but could not make them work for $n=50$.
Figure~\ref{fig:zeros} depicts these 47 zeros in blue over the eigenvalues in red, and we note that 2 eigenvalues are outside of the plot range, both being large negative numbers. 

It seems reasonable to assume that our randomly generated PH is likely to be minimal (it is such for smaller $n$), and so various zeros are not detected.
It is not possible to directly verify this claim by plugging the eigenvalues into $\psi$, because the latter is often extremely sensitive at the negative values. The missing zeros may not influence $W_q$ much due to their magnitudes and the magnitudes of the corresponding derivatives, and yet numerical stability is always a concern for this method.
In this example the two methods produce almost identical scale functions, see Figure~\ref{fig:scale}.
Importantly, computing $\mat G$ and finding its eigenvalues is much faster (0.003 vs 13 seconds in this example) and unlike the current implementation of the root finding method our algorithm can easily handle $n=100$ and beyond.

\begin{figure}[h!]
\centering
\includegraphics[width=0.45\textwidth]{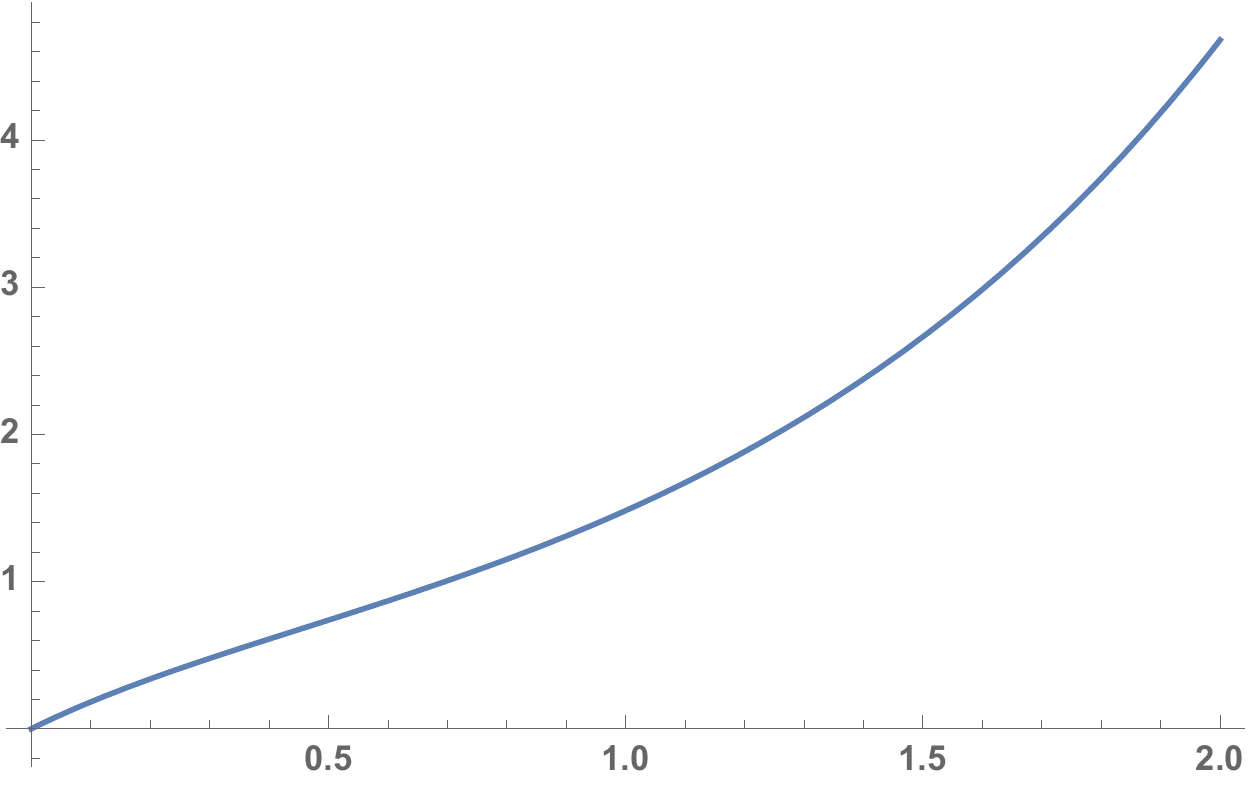}\quad
\includegraphics[width=0.45\textwidth]{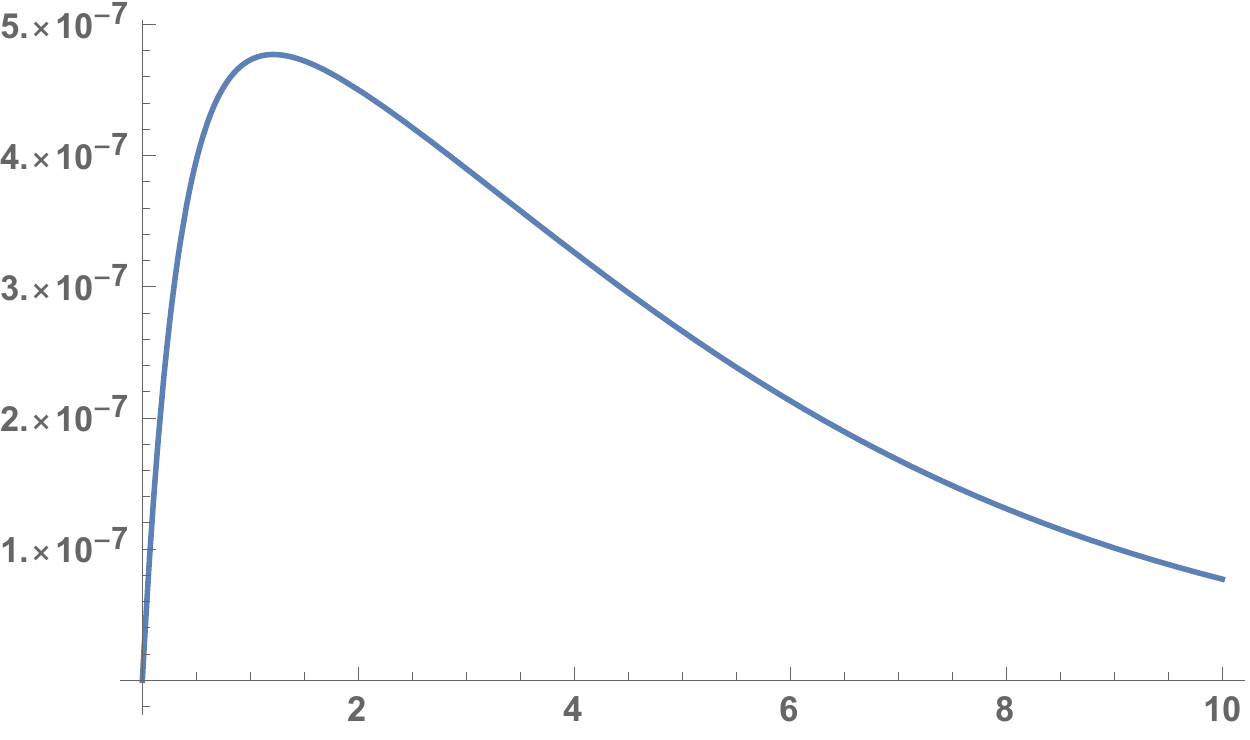}
\caption{The scale function $W_1(x)$ (left) and the absolute error between such functions obtained by the two methods (right).}
\label{fig:scale}
\end{figure}

\subsection{Convergence}
Here we illustrate convergence of the iterative scheme in Proposition~\ref{prop:bm} for $q=1$ and $q=0$.
We find $a=3.1$ and $a=2.33$ in the two cases, respectively. 
For the initial $a_0$ we examine three choices as suggested by the interval in~\eqref{eq:interval}:
(i) $2.001$ which is almost at the lower boundary, (ii) the upper boundary and (iii) the midpoint.
As mentioned above, the algorithm stops when $|a_n-a_{n-1}|<10^{-5}$.
Recall that in the case $q=0$ our algorithm must not be started at $2$, since then it stays there.
When starting at $2.001$ our algorithm will take some time to escape that fixed point as can be seen in Figure~\ref{fig:convergence} (left blue).
The Q-linear convergence rates are $0.57$ for $q=0$ and $0.1$ for $q=1$.
Thus we expect a much faster termination in the case $q=1$ after $a_n$ becomes close to the true $a$, which is confirmed by Figure~\ref{fig:convergence}.
\begin{figure}[h!]
\centering
\includegraphics[width=0.45\textwidth]{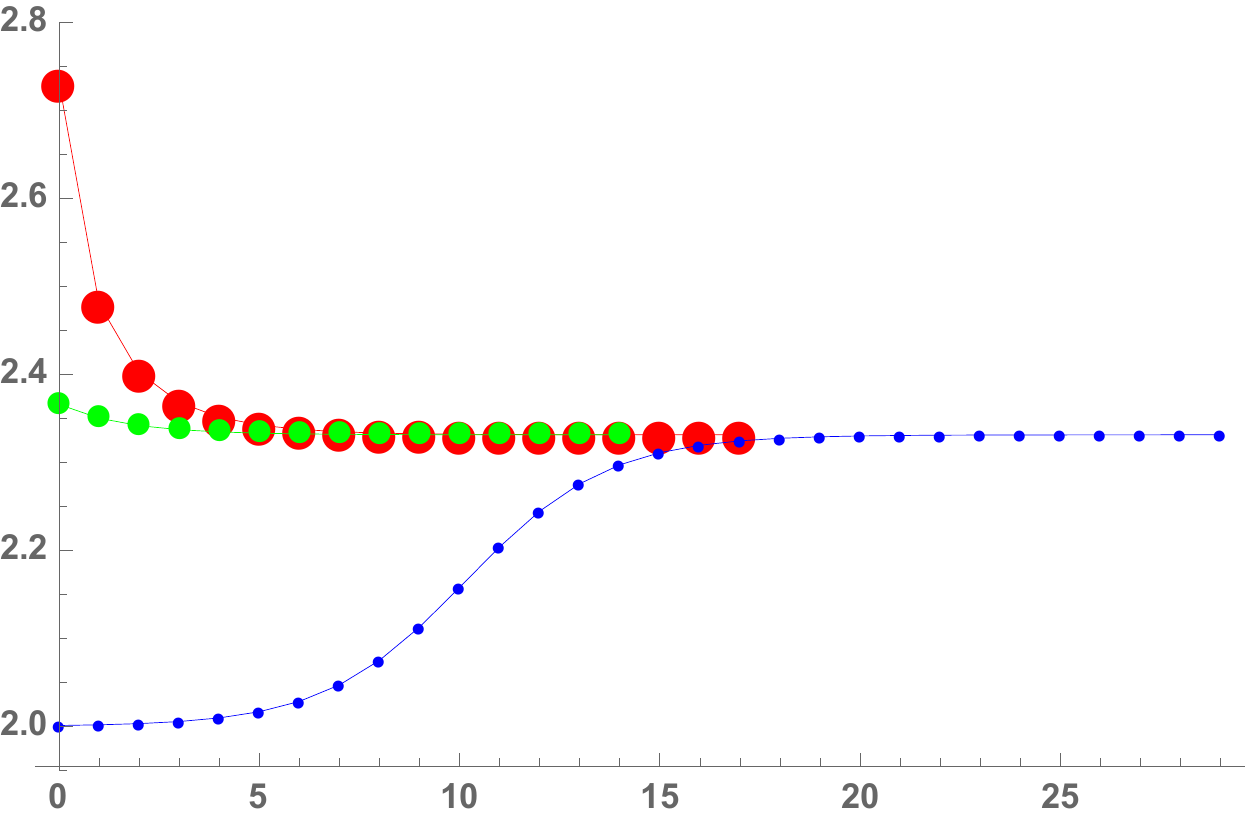}\quad
\includegraphics[width=0.45\textwidth]{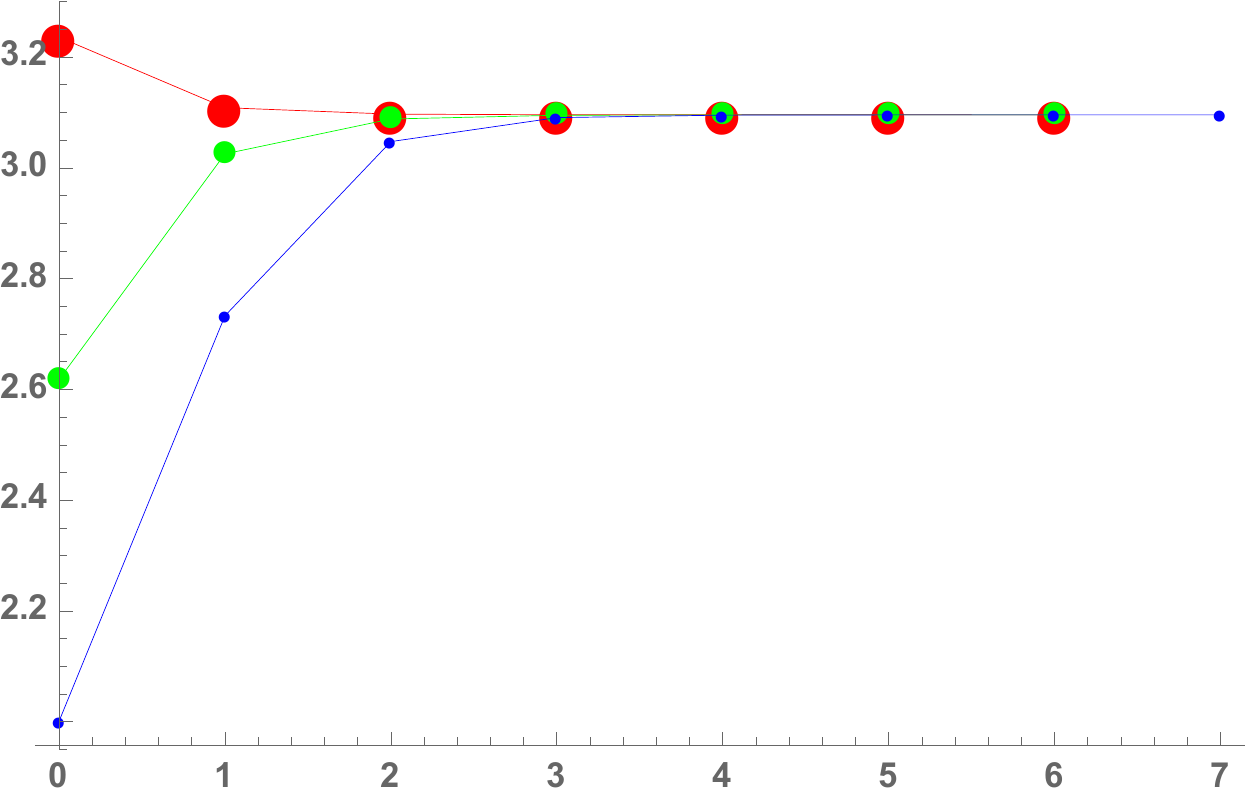}
\caption{Convergence of $a_n$ for $q=0$ (left) and for $q=1$ (right) until the difference is below $10^{-5}$.}
\label{fig:convergence}
\end{figure}

We have also tried the recursion based on~\eqref{eq:ab2}, but that resulted in non-monotone sequences (alternating in the sign of increments) exhibiting much slower convergence.

\subsection{On the number of iterations}
Finally, we randomly generate 1000 PH distributions as specified above and compute the corresponding $a$ by starting at the midpoint (case (iii), green in Figure~\ref{fig:convergence}).
\begin{figure}[h!]
\centering
\includegraphics[width=0.45\textwidth]{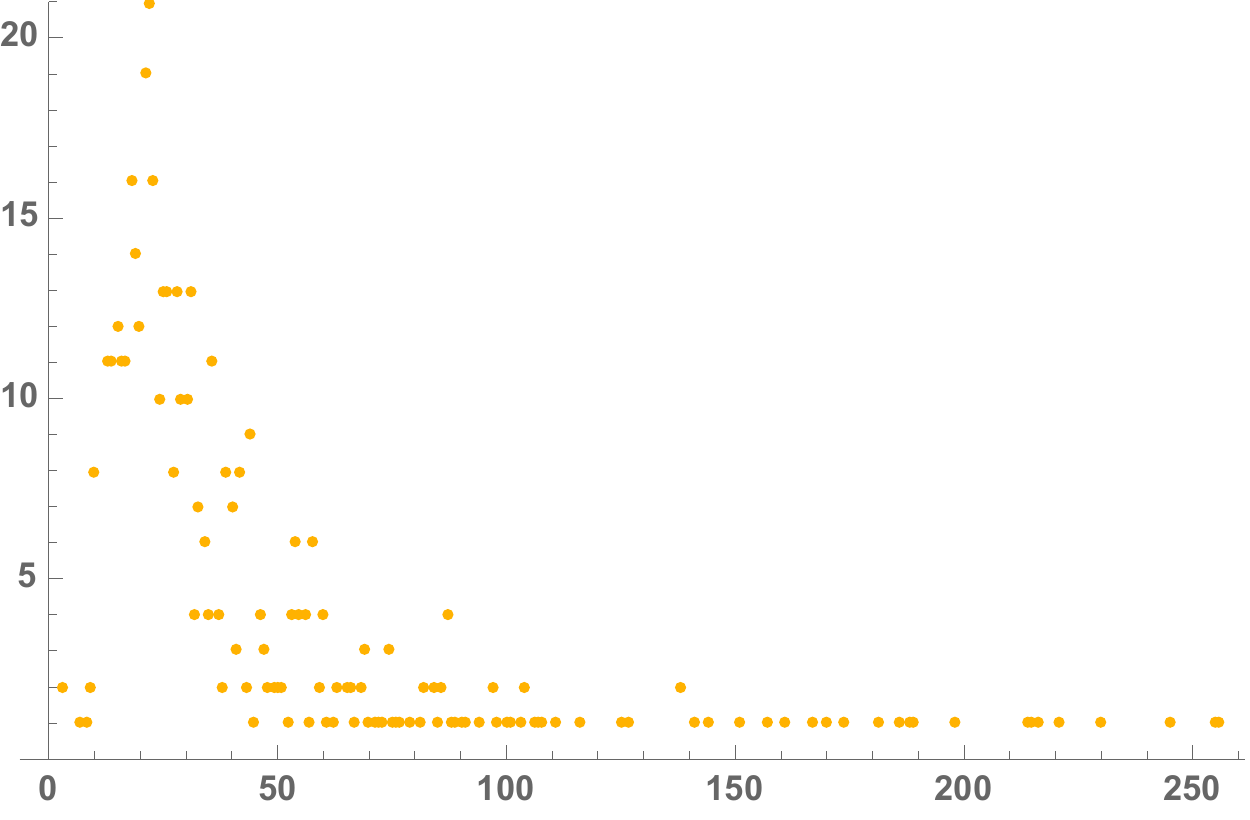}\quad
\includegraphics[width=0.25\textwidth]{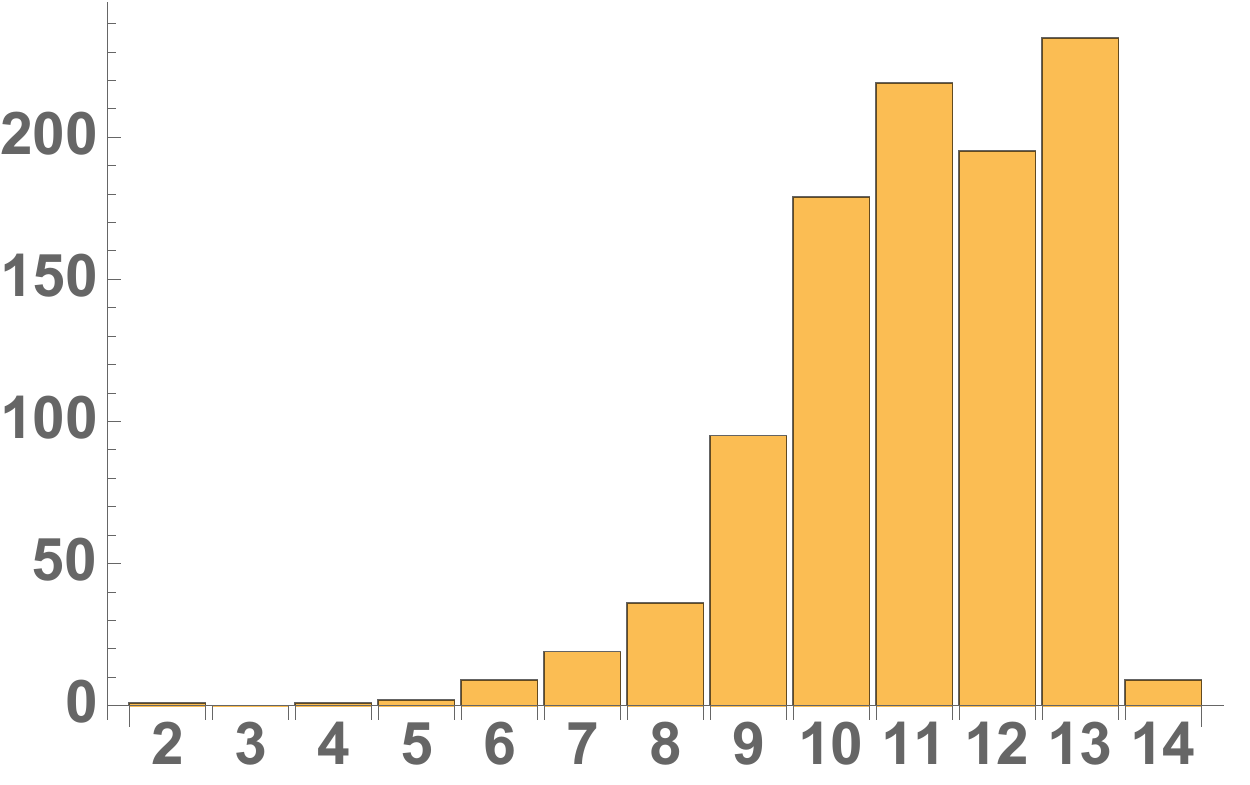}
\includegraphics[width=0.25\textwidth]{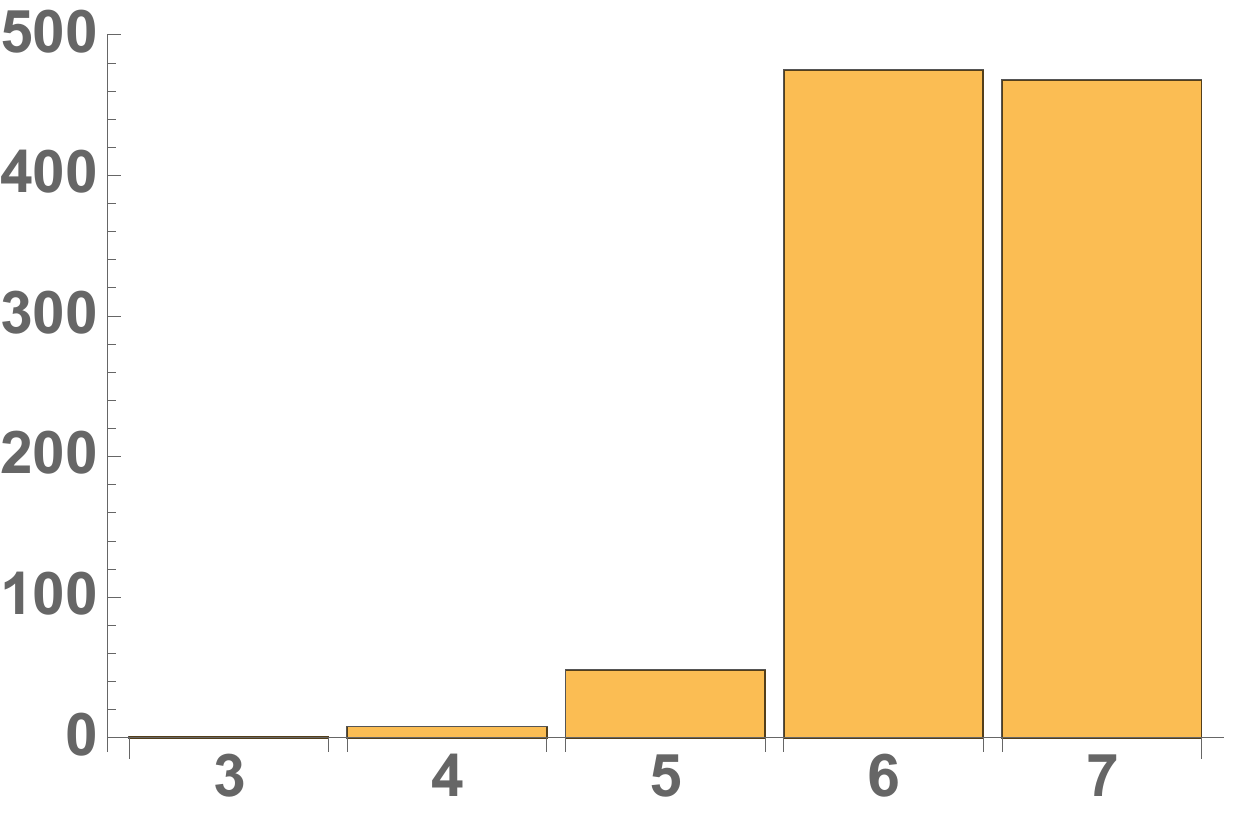}
\caption{Counts for the number of iterations among 1000 replications. Left: $q=0$ with 470 cases  corresponding to $\e X_1>0$ excluded. Right: $q=0.1$ and $q=1$.}
\label{fig:iterations}
\end{figure}
 The counts for the number of iterations for $q=0,q=0.1,q=1$ are presented in Figure~\ref{fig:iterations}, where in the case $q=0$ we have excluded the realizations with $\e X_1>0$ (470 of such) corresponding to an explicit~$a$.
The execution times are 4.6 sec.\ for $q=0$ (530 problems), 2.7 sec.\ for $q=0.1$ and 1.7 sec.\ for $q=1$ (1000 problems in both).
For $n=100$ and $q=1$ the histogram is almost the same as when $n=50$.

Additionally, we compute the Q-linear convergence rates for each generated model and provide the histograms in Figure~\ref{fig:hist}. We get mean 0.71 and maximum 0.99 for $q=0$, mean 0.41 and maximum 0.5 for $q=0.1$, mean 0.11 and maximum 0.15. This is in agreement with the fact that the convergence rate is decreasing in~$q$.
Thus the numerical problem of finding $a,\mat b$ becomes simpler for larger values of the killing rate $q\geq 0$. 
\begin{figure}[h!]
\centering
\includegraphics[width=0.3\textwidth]{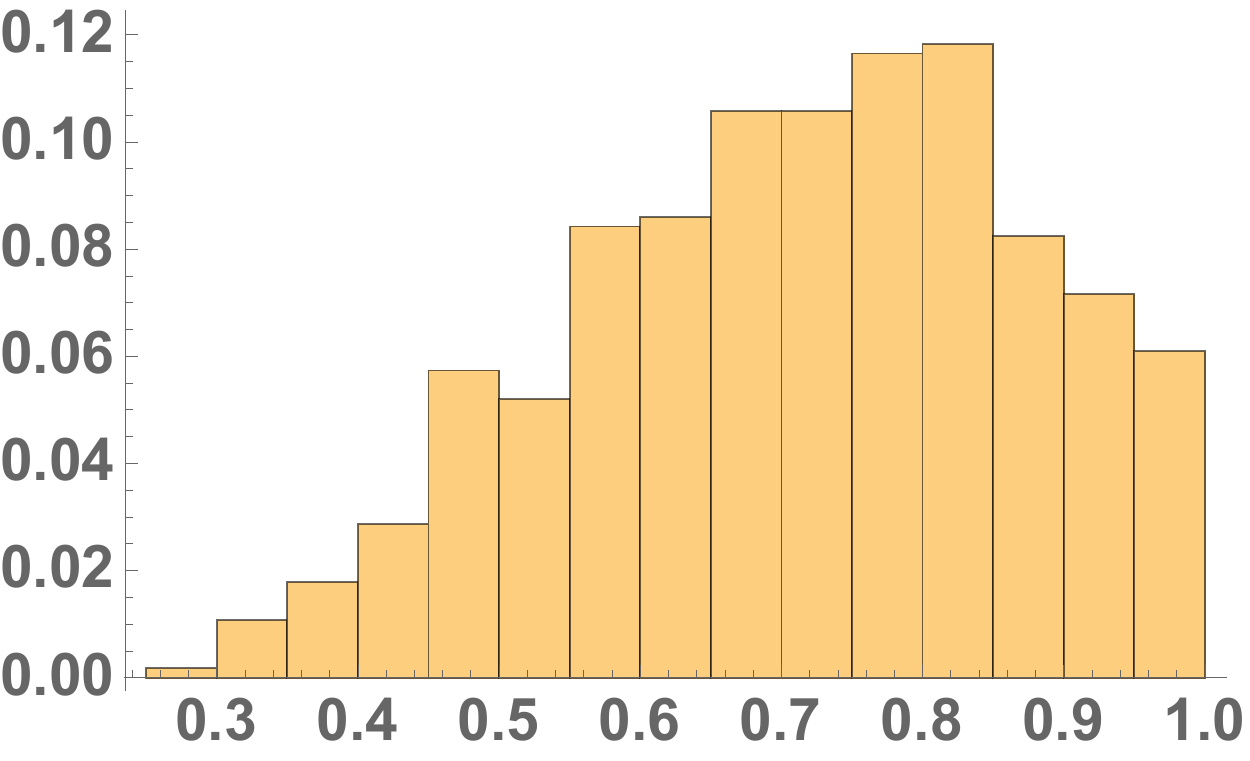}
\includegraphics[width=0.3\textwidth]{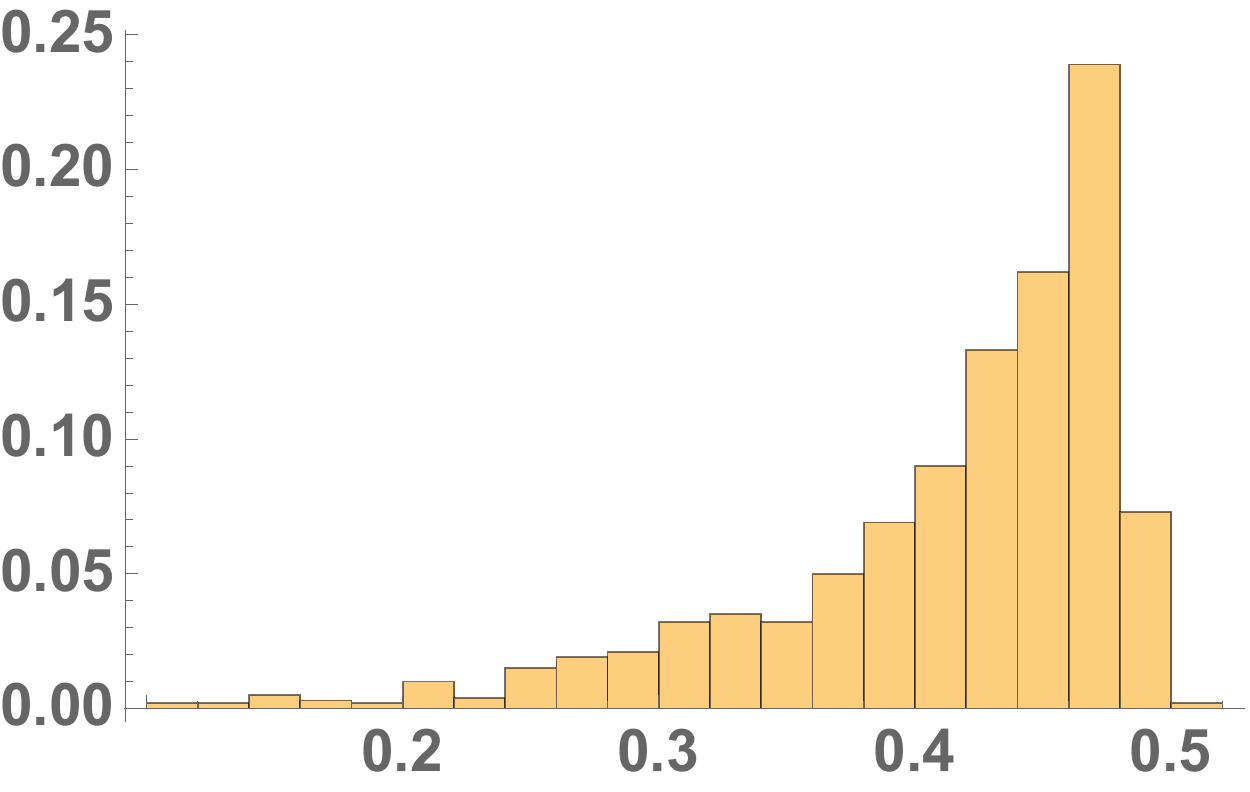}
\includegraphics[width=0.3\textwidth]{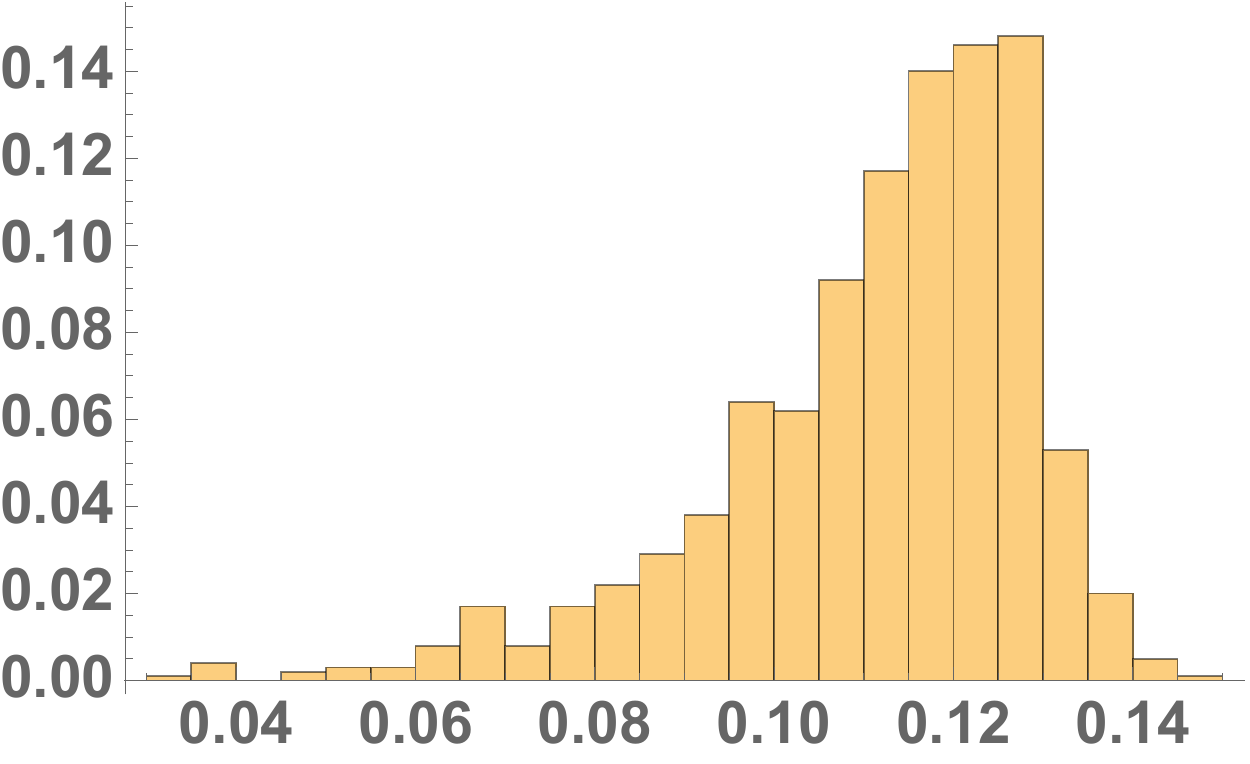}
\caption{Histograms of the convergence rate. Left to right: $q=0$ (530 cases), $q=0.1$ and $q=1$ (1000 cases both).}
\label{fig:hist}
\end{figure}

In conclusion, the monotone iterative scheme in Proposition~\ref{prop:bm} performs very well for large $n$ in our setting of randomly generated PH distributions of Coxian type.
Importantly, it is trivial to implement, and one can easily provide error guarantees by taking advantage of monotone convergence.
Finally, the resulting exponential form of $W_q$ in Theorem~\ref{thm:bm} is convenient for various further calculations.

\section*{Acknowledgments}
The financial support of Sapere Aude Starting Grant 8049-00021B ``Distributional Robustness in Assessment of Extreme Risk'' 
 is gratefully acknowledged. 
 

\end{document}